\documentclass[11pt]{article}

\usepackage[a4paper, left=3.2cm,top=3.2cm]{geometry}
\usepackage{authblk}
\usepackage[onehalfspacing]{setspace}

\usepackage{amsmath,amsfonts}
\usepackage{algorithmic}
\usepackage{algorithm}
\usepackage{array}
\usepackage[caption=false,font=normalsize,labelfont=sf,textfont=sf]{subfig}
\usepackage{textcomp}
\usepackage{stfloats}
\usepackage{url}
\usepackage{verbatim}
\usepackage{graphicx}
\usepackage{cite}
\usepackage{cite}
\usepackage{mathtools}
\usepackage{enumitem}
\usepackage{bm}
\usepackage{amssymb}

\newcommand{\To}{\mathbf{T}}
\newcommand{\Uo}{\mathbf{U}}
\newcommand{\Do}{\bm{\Phi}}

\newcommand{\Mo}{\mathbf{M}}
\newcommand{\Lo}{\mathbf{L}}
\newcommand{\Ao}{\mathbf{A}}
\newcommand{\Ro}{\mathbf{R}}
\newcommand{\Ho}{\mathbf{H}}
\newcommand{\Ko}{\mathbf{K}}
\newcommand{\ft}{\mathcal{F}}

\newcommand{\dis}{\mathcal{D}}

\newcommand{\reg}{\mathcal{R}}

\newcommand{\Tc}{\mathcal{T}}
\newcommand{\Lc}{\mathcal{L}}

\DeclareMathOperator*{\argmin}{arg\, min}
\DeclareMathOperator*{\argmax}{arg\, max}

\newcommand{\dom}{\operatorname{dom}}
\newcommand{\id}{\operatorname{Id}}
\newcommand{\ran}{\operatorname{ran}}
\newcommand{\im}{\operatorname{im}}
\newcommand{\hard}{\mathbf H}
\newcommand{\soft}{\mathbf S}

\newcommand{\X}{\mathbb{X}}
\newcommand{\Y}{\mathbb{Y}}
\newcommand{\N}{\mathbb{N}}
\newcommand{\R}{\mathbb{R}}
\newcommand{\Z}{\mathbb{Z}}

\newcommand{\sph}{\mathbb{S}}

\newcommand\norm[1]{\lVert#1\rVert}
\newcommand\set[1]{\{#1\}}
\newcommand\abs[1]{\lvert#1\rvert}

\newcommand\migabs[1]{\biggl\lvert#1\biggr\rvert}

\newcommand\inner[2]{\langle#1,#2\rangle}

\newcommand\miginner[2]{\biggl\langle#1,#2\biggr\rangle}

\newcommand{\la}{\lambda}

\newcommand{\ga}{\gamma}

\newcommand{\al}{\alpha}

\newcommand{\eps}{\epsilon}

\newcommand{\prox}{\operatorname{prox}}
\newcommand{\ffix}{\operatorname{fix}}
\newcommand{\fix}{\operatorname{Fix}}
\newcommand{\PnP}{\operatorname{PnP}}
\newcommand{\lip}{\operatorname{Lip}}
\newcommand{\sign}{\operatorname{sign}}

\usepackage{amsthm}
\newtheorem{maintheorem}{Theorem}

\newtheorem{theorem}{Theorem}
\newtheorem{definition}[theorem]{Definiton}
\newtheorem{corollary}[theorem]{Corollary}
\newtheorem{lemma}[theorem]{Lemma}
\newtheorem{assumption}[theorem]{Assumption}
\newtheorem{proposition}[theorem]{Proposition}
\newtheorem{remark}[theorem]{Remark}
\newtheorem{example}[theorem]{Example}

\numberwithin{equation}{section}
\numberwithin{theorem}{section}

\allowdisplaybreaks

\title{Plug-and-Play image reconstruction is a convergent regularization method}

\date{December 13, 2022}

\author{Andrea Ebner}

\affil{Department of Mathematics, University of Innsbruck\authorcr
Technikerstrasse 13, 6020 Innsbruck, Austria
 \authorcr E-mail:  \texttt{andrea.ebner@uibk.ac.at}
 }

\author{Markus Haltmeier}

\affil{Department of Mathematics, University of Innsbruck\authorcr
Technikerstrasse 13, 6020 Innsbruck, Austria
 \authorcr E-mail:  \texttt{markus.haltmeier@uibk.ac.at}
 }

\begin{document}

\maketitle

\begin{abstract}
Non-uniqueness and  instability are characteristic features of image reconstruction processes. As a result, it is necessary to develop regularization methods that can be used to compute reliable approximate solutions. A regularization method provides of a family of stable reconstructions that converge to an exact solution of the noise-free problem as the noise level tends to zero. The standard regularization technique is defined by variational image reconstruction, which minimizes a data discrepancy augmented by a regularizer. The actual numerical implementation makes use of iterative methods, often involving proximal mappings of the regularizer. In recent years, Plug-and-Play image reconstruction (PnP)  has been developed as a new powerful generalization of variational methods based on replacing proximal mappings by more general image denoisers. While PnP iterations yield excellent results, neither stability nor convergence in the sense of regularization has been studied so far. In this work, we extend the idea of PnP by considering families of PnP iterations, each being accompanied by its own denoiser.  As our main theoretical result, we show that such PnP reconstructions lead to stable and convergent regularization methods. This shows for the first time that PnP is mathematically equally justified for robust image reconstruction as variational methods.

\medskip\noindent\textbf{keywords}
Regularization, Plug-and-Play, image prior, convergence analysis, stability, inverse problems, ADMM, forward backward splitting

\end{abstract}

\section{Introduction}

Image reconstruction is an essential step for a variety of important applications including  biomedical imaging, non-destructive testing and remote sensing. These problems are often formulated as an inverse problem of the form
\begin{equation}\label{eq:ip}
	\text{Estimate $x$  from } \quad y^\delta = \Ao x + z^\delta \,.
\end{equation}
Here $x \in \X$ is the image to be recovered, $\Ao \colon \dom(\Ao) \subseteq \X \to \Y$ is an operator between Hilbert spaces modeling the  forward problem, $z^\delta$ denotes the data perturbation and $y^\delta$ are the given noisy data. Throughout we assume a deterministic error model where   we assume the estimate $\norm{z^\delta} \leq \delta$ with noise level $\delta \geq  0$. For vanishing noise level $\delta =0$ we refer to $\Ao x$ as exact data.

Most image reconstruction  problems suffer from  non-uniqueness and instability.  To overcome both issues  regularization methods have to applied. A regularization method consists of a family $(\Ro_\la)_{\la >0}$ of continuous reconstruction mappings $\Ro_\la \colon \Y \to \X$ that are convergent  in the sense that for a suitable parameter choice  $\la = \la(\delta, y^\delta)$ we have $\norm{ \Ro_{\la(\delta, y^\delta)} (y^\delta) - x} \to 0$ as $\delta \to 0$ uniformly in all $y^\delta$ with $\norm{y^\delta- \Ao x} \leq \delta$. The main goal of this work is to uncover Plug-and-Play (PnP) image reconstruction as a new member of the class of regularization methods.

\subsection{Variational regularization}

To put  our results  into perspective we start with variational image reconstruction that is probably the most established regularization concept. It defines  near-solutions $x^\delta_\al = \Ro_\la (y^\delta)$ of \eqref{eq:ip} as minimizers of the generalized Tikhonov functional
$\Tc_{y^\delta,\la} (x) \triangleq \norm{\Ao x  - y^\delta }^2/2  +  \la \reg(x)$. The data discrepancy term $\norm{\Ao x  - y^\delta }^2/2$  penalizes the distance between the predicted data $\Ao x $  and the measured data $y^\delta$, making minimizers of the Tikhonov functional being a near-solution of  \eqref{eq:ip}.  The regularizer $\reg$ on the other hand includes prior information thereby enforcing stability and regularity of the recovered image. The regularization parameter $\la >0$ allows to balance between data discrepancy and regularity and can be  selected depending on the noise level and other available prior information.
Benefits of variational regularization  are  clear interpretation and the well-developed theoretical understanding.  While classically hand-crafted regularizers  such as quadratic penalties, the $\ell^1$ norm or total variation  \cite{engl1996regularization,scherzer2009variational,acar1994analysis,daubechies2004iterative} are utilized, in recent works the use of trained neural networks has also been proposed \cite{li2020nett,obmann2021augmented,lunz2018adversarial}.

The realization of variational regularization requires iterative algorithms for  minimizing the Tikhonov functional. A particular flexible algorithm for minimizing  $\Tc_{y^\delta,\la}$ is the forward-backward splitting (FBS) iteration  $x_{\la, n+1}^{\delta} = \prox_{s \la \reg} (x_{\la, n}^{\delta}  -  s \Ao^*(\Ao x_{\la, n}^{\delta} - y^\delta) )$ where $\prox_{s \la \reg} \triangleq (\id_\X + s \la \nabla  \reg)^{-1}$  is the proximal mapping of the scaled regularizer $s \la \reg$ and $s>0$ is the step size acting as scaling parameter. In the case that $\reg$ is strongly convex and the step size is chosen properly then convergence  of  the FBS iterates towards a minimizer of   $\Tc_{y^\delta,\la}$  follows  from the Banach fixed point theorem.
Convergence of FBS iterations also holds for more general regularizers \cite{combettes2005signal,daubechies2004iterative}. Other common algorithms for minimizing the Tikhonov functioal include ADMM \cite{glowinski2014alternating} or primal dual \cite{chambolle2011first} algorithms.

\subsection{Plug-and-Play image reconstruction}

The proximal operator $\prox_{s \la \reg}$   used in iterative image reconstruction algorithms can be seen as specific denoiser defined by a regularizer. The basic idea  of Plug-and-Play (PnP) image reconstruction \cite{Ve13}, is  to select a proximal algorithm and to  replace the proximal operator with a more general denoiser  $\Do(\la,\cdot) \colon \X \to \X$.  In particular, such a strategy has been implemented for FBS  \cite{Ry19}, ADMM \cite{Sr15, Wo19}  and primal dual methods \cite{Me17}. In this paper we focus on PnP-FBS which for \eqref{eq:ip} reads
 \begin{equation} \label{eq:PnP}
    x_{\la, n+1}^\delta = \Do(\la,\cdot) \circ (x_{\la, n}^{\delta}  -  s \Ao^*(\Ao x_{\la, n}^{\delta} - y^\delta) )  \,.
\end{equation}
Clearly  FBS  is a special case by taking $\Do(\la,\cdot) =  \prox_{s\la \reg}$.
However, PnP-FBS allows to include other state-of-the-art denoisers such as BM3D \cite{dabov2007image} or  trained CNNs  \cite{xie2012image}.
PnP offers greater flexibility than the variational approach. First, a potential variational characterization of the a denoiser as proximal mapping might not be known. Second, and more importantly, any proximal  mapping is in particular the gradient of some functional which excludes any denoiser that is not of gradient form. Hence PnP significantly extends the class of variational image reconstruction.

One theoretical question in the context of PnP is convergence of the iteration \eqref{eq:PnP}.  Due to the fixed point structure,  convergence can  be derived from the large pool of existing fixed point theorems guaranteeing convergence to a fixed point of  $\Do(\la,\cdot) \circ (\id_\X  -  s \Ao^*(\Ao (\cdot) - y^\delta) )$. However, main theoretical question such as stability and convergence of fixed points as $\delta \to 0$ has not been addressed so far. This gap will be closed in this paper.

\subsection{Main results}

As one main result of this paper we show that PnP fixed points define a stable and  convergent regularization method. To the best of our knowledge,  regularization theory of PnP has not been established so far. For that  purpose  we consider an admissible family of denoisers  $(\Do(\la,\cdot))_{\la>0}$ (see Definition \ref{def:denoisers} below for precise terminology), each associated with a corresponding PnP iteration and associated fixed point
\begin{equation} \label{eq:PnP-LS}
	\PnP(\la, y^\delta) \triangleq \fix ( \Do(\la,\cdot) \circ (\id_\X  -  s \Ao^*(\Ao (\cdot) - y^\delta) ))  \,.
\end{equation}
Under conditions specified below we show  that  \eqref{eq:PnP-LS} defines  a convergent regularization method for \eqref{eq:ip}. Furthermore, we derive a characterization of the limiting problem extending the concept of $\reg$-minimizing solutions known from  variational regularization.

The following theorems are  the  main result of this  paper. In that theorems  $(\Do(\la,\cdot))_{\la >0}$ is a family  of admissible denoisers, that we rigorously define  in Section~\ref{sec:analysis}. In that section we also give more complete statements in a more general setting including the standard inverse Problem~\eqref{eq:ip} as special case.

\begin{maintheorem}[Stability] \label{thm:stability}
For all regularization parameters $\la >0$, $\PnP(\la,\cdot)$ is continuous.
\end{maintheorem}

\begin{proof}
See Section \ref{sec:stability}.
\end{proof}

\begin{maintheorem}[Convergence] \label{thm:convergence}
For any $y \in \ran(\Ao)$ and any sequence $(\delta_k)_{k \in \N} \in (0,\infty)^\N$  of noise levels converging to $0$ there  exists a sequence $(\la_k)_{k \in \N} \in (0,\infty)^\N$  of regularization parameters converging to $0$  such that for all $y_k \in \Y$ with $\norm{y-y_k} \leq \delta_k$ the following hold:
\begin{enumerate}[label=(\alph*)]
	\item $(\PnP(\la_k,y_k))_{k\in \N}$ has a weakly convergent subsequence.
	\item  The limit of every weakly convergent subsequence of $(\PnP(\la_k,y_k))_{k\in \N}$ is a solution of \eqref{eq:ip}.
\end{enumerate}
\end{maintheorem}

\begin{proof}
See Section \ref{sec:convergence}.
\end{proof}

Theorems  \ref{thm:stability} and \ref{thm:convergence}  imply that $(\PnP(\la,\cdot))_{\la >0} $  is a regularization method for  \eqref{eq:ip}. Actually these results are derived in a more general framework (introduced in Subsection~\ref{ssec:illposed}) that includes  \eqref{eq:ip} as a special  case.

\begin{maintheorem}[PnP limits]\label{thm:char}
In the situation of Theorem \ref{thm:convergence} suppose $ \Do(\la,\cdot)^{-1}$ is single valued and that $( \Do(\la,\cdot)^{-1}-\id)/\al$  converges weakly uniformly on bounded sets to a weakly continuous $\Ho \colon \X \to \X$. Then any weak accumulation point $x^*$ of $(\PnP(\la_k,y_k))_{k\in \N}$  satisfies
$\Ao (x^*) = y $ and $\Ho(x^*)  \in  \ker(\Ao)^\perp$.
\end{maintheorem}

\begin{proof}
See Section \ref{sec:char}.
\end{proof}

\subsection{Notation}

Operator  $\To \colon \X \to \X$ is called $L$-Lipschitz continuous with  $L > 0$  if $\norm{\To(x_1)-\To(x_2)} \leq L \norm{x_1-x_2}$ for all $x_1, x_2 \in \X$.  By $\lip(\To)$ we denote the smallest possible Lipschitz constant of $\To$. If $\lip(\To)\leq 1$  then $\To$ is called non-expansive and if $\lip(\To)< 1$  then it is called contraction.  Finally, $\To$ is called firmly non-expansive if $\norm{\To(x_1)-\To(x_2)}^2
 \leq \inner{x_1-x_2}{\To(x_1)-\To(x_2)}$ for all $x_1,x_2\in \X$.
Recall $x \in \X$ is called a fixed point of $\To$ if $\To(x) = x$. We write  $\fix(\To) = \{x \in \X \mid \To(x) = x\}$ for the set of all fixed points. If  $\fix(\To)$ consists of a single element we denote this  element by $\ffix (\To)$.

A family $(\To_\la)_{\la > 0}$  of mappings on $\X$ is said to pointwise strongly convergence to $\To$  if  $\lim_{\la \to  0} \norm{\To_\la(x) - \To(x)} = 0$ for all $x \in \X$.
It is said to converge  weakly uniformly on some set $B \subseteq \X$  if $\forall z\colon \sup_{x \in B} \inner{\To_\la(x)-\To(x)}{z} \to 0$.

Let  $\Gamma_0(\X) $ denote the set  of all $\reg \colon \X \to [0, \infty]$ that are proper, convex and lower semicontinuous.  For $\reg \in \Gamma_0(\X) $ the proximal mapping  $\prox_{\reg} \colon  \X \to \X$ is
uniquely defined by $\prox_{\reg} (x) \coloneqq  \argmin_{z \in \X}  \norm{x-z}^2/2 + \reg(z) $ and firmly nonexpansive \cite[Prop. 12.28]{bauschke2011convex}.

\section{Background}

Throughout this paper, $\X$ and $\Y$ denote possible infinite dimensional Hilbert spaces equipped with the inner product  $\inner{\cdot}{\cdot}$ and induced norm topology.

\subsection{Ill-posed minimization problems}
\label{ssec:illposed}

In the exact data situation, solving \eqref{eq:ip} reduces to the solution of the equation $\Ao x = y$. Equivalently this amounts minimizing the least squares functional $ \norm{\Ao x- y}^2/2$. Throughout this paper we consider the more general situation where we look for minimizers of some discrepancy  functional $ \dis \colon \Y \times \X \to [0, \infty)$ with respect to the variable $x$. We thus consider
\begin{equation}\label{eq:min}
	\text{Estimate $\argmin_x \dis(\cdot,y)$  from  $y^\delta = y + z^\delta$} \,.
\end{equation}
Here $y, y^\delta \in \Y$  are exact and noisy data, respectively, and  we assume the known  noise bound $\norm{z^\delta}   \leq \delta$. With the particular choice $\dis(x,y) =   \norm{\Ao x -y}^2/2$, Problem  \eqref{eq:min} reduces to the standard image reconstruction Problem \eqref{eq:ip}. One may take $\dis(x,y) = \Lc (\Ao x, y)$ for a loss function $\Lc$. However, our framework also includes cases where no forward operator $\Ao$ can be identified or where $\Ao$ is defined implicitly via minimization.

Throughout the paper we will make the  following assumptions on the discrepancy functional.

\begin{assumption}[Assumptions on $\dis$]\label{ass:A}\mbox{}
\begin{enumerate}[label=(A\arabic*), leftmargin=2.5em]
\item \label{ass:A1} $\dis$ is convex and Fr\'echet differentiable in $x$.
\item \label{ass:A2} $\nabla_x \dis$ is $1/\beta$-Lipschitz continuous in $x$ with $\beta > 0$.
\item \label{ass:A3} $\nabla_x \dis$ is weakly sequentially continuous  in $x$
\item \label{ass:A4} $(\nabla_x \dis (x,\cdot))_{x \in \X}$ is equicontinuous.
\end{enumerate}
\end{assumption}

Recall that $(\nabla_x \dis (x,\cdot))_{x \in \X}$ is  equicontinuous if $\norm{y_k - y} \to 0$ implies $ \sup_{x \in \X} \norm{\nabla_x \dis (x,y_k) - \nabla_x \dis (x,y)} \to  0$  .

\begin{example}[Least squares]
Let us verify that \ref{ass:A1}-\ref{ass:A4} are indeed satisfied for the standard inverse Problem \eqref{eq:ip} with bounded linear $\Ao \colon \X \to \Y$ and associated least squares  functional $\dis_{\rm LS}(x,y) \triangleq \norm{\Ao x -y }^2 / 2$.
Clearly $\dis_{\rm LS}$ is  convex and Fr\'echet differentiable in  $x$ with gradient $\nabla_x \dis_{\rm LS}(x,y) = \Ao^* (\Ao x -  y)$. Further, $\nabla_x \dis_{\rm LS}(\cdot,y) $ is $\norm{\Ao}^2$-Lipschitz continuous and weakly continuous. Finally, the estimate  $\norm{\nabla_x \dis_{\rm LS}(x,y_1) - \nabla_x \dis_{\rm LS}(x,y_2)}  \leq \norm{\Ao} \norm{y_1-y_2} $ shows that  $(\nabla_x \dis_{\rm LS} (x,\cdot))_{x \in \X}$  is equicontinuous.
\end{example}

Under \ref{ass:A1} the following fixed point characterization holds.

\begin{proposition}[Exact data] \label{prop:exact}
Let \ref{ass:A1} be satisfied. For all $(x^0,y^0) \in \X \times \Y$ and $s>0$ the following are equivalent:
\begin{enumerate}[label=(\roman*)]
\item\label{lem:exact1}  $x^0 \in \argmin \dis(\cdot,y^0)$.
\item\label{lem:exact2}  $\nabla_x \dis (x^0,y^0) = 0$.
\item\label{lem:exact3} $x^0 \in \fix(\id_\X - s \nabla_x \dis (\cdot,y^0))$
\end{enumerate}
If these conditions are satisfied for at least one $x^0$ then we refer to $y^0$ as exact data.
\end{proposition}

\begin{proof}
Since is $\dis(\cdot,y)$ is convex and Fr\'echet differentiable, by Fermat's rule \cite[Prop. 17.4]{bauschke2011convex}, $x^0$ is a minimizer of $\nabla_x\dis(\cdot,y^0)$ if and only if the first order optimality condition $\nabla_x \dis(x^0,y^0) = 0$ holds which is the equivalence of \ref{lem:exact1} and  \ref{lem:exact2}. By elementary reformulation, the optimality condition is equivalent to $x^0 = x^0 - s \nabla_x \dis(x^0,y^0)$ which   is  \ref{lem:exact3}.
\end{proof}

\begin{lemma}\label{lem:grad-step}
If \ref{ass:A1}, \ref{ass:A2} hold  then  $\id_\X - s \nabla_x \dis(\cdot,y)$ is non-expansive for all $y \in \Y$ and $s \in (0,2\beta)$.
\end{lemma}

\begin{proof}
Since $\dis(\cdot,y)$ is convex, Fr\'echet differentiable and $\nabla_x \dis(\cdot,y)$ is $1/\beta$-Lipschitz, by the Baillon-Haddad theorem \cite{bauschke2010baillon,Ba77}, the operator $\beta \nabla_x \dis(\cdot,y)$ is firmly non-expansive.  Thus for  all $x_1, x_2 \in \X$ we have
\begin{align*}
&\norm{ (x_1-s \nabla_x \dis(x_1,y)  - (x_2-s \nabla_x \dis(x_2,y)}^2 \\
&= \norm{x_1-x_2}^2 + s^2 \norm{\nabla_x \dis(x_1,y) - \nabla_x \dis(x_2,y)}^2 \\
&\qquad - 2 s \inner{x_1-x_2}{\nabla_x \dis(x_1,y) - \nabla_x \dis(x_2,y)} \\
&\leq \norm{x_1-x_2}^2 + s^2 / \beta  \inner{x_1-x_2}{\nabla_x \dis(x_1,y)-\nabla_x \dis(x_2,y)}
\\ &\qquad- 2 s \inner{x_1-x_2}{\nabla_x \dis(x_1,y) - \nabla_x \dis(x_2,y)} \\
&\leq \norm{x_1-x_2}^2 \,,
\end{align*}
where we used that  $s^2 \leq 2 \beta s$.
\end{proof}

Image reconstruction problems are commonly ill-posed, which means that minimizers of $\dis(\cdot,y^\delta)$ do not  exist, are not unique, or depend unstably on data $y^\delta$.
In order to account for the ill-posedness  one has  to apply regularization methods \cite{scherzer2009variational}. In this paper we use the following notion of a weakly convergent regularization method.

\begin{definition}[Regularization method] \label{def:reg}
A family $(\Ro_\la)_{\la >0}$ of mappings $\Ro_\la \colon \Y \to \X$ together with a parameter choice rule $\la^* \colon (0, \infty) \to (0, \infty)$ where $\lim_{\delta \to 0}\la^*(\delta)=0$ is called regularization method for $\argmin_x \dis(x,y)$  over  $E \subseteq \Y$ if the following holds:
\begin{enumerate}[label=(\alph*)]
\item Stability:  $ \forall \la >0 \colon \Ro_\la$ is continuous.
\item  Convergence: For all $y \in E$,  $(\delta_k)_{k \in \N} \in (0,\infty)^\N$  converging to $0$  and all $y_k \in \Y$ with $\norm{y-y_k} \leq \delta_k$, the   sequence  $(\PnP(\la^*(\delta_k),y_k))_{k\in \N}$ has a weakly convergent subsequence and the limit of every weakly convergent subsequence of $(\PnP(\la (\delta_k),y_k))_{k\in \N}$ is a minimizer of  $\dis(\cdot,y)$.
\end{enumerate}
\end{definition}

The aim of this paper is to uncover PnP as a regularization method for $\argmin_x \dis(x,y)$.

\subsection{The PnP method}

Variational image reconstruction is probably the most established  regularization method, at least for the standard inverse Problem \eqref{eq:ip}. Based on  a regularizer $\reg \colon \X \to [0, \infty)$, $\Ro_\la(y^\delta)$ for the more general Problem \eqref{eq:min}, it is defined  as a minimizer of the Tikhonov functional $\dis(\cdot,y^\delta) + \la \reg$. The following  characterization of variational regularization generalizing Proposition~\ref{prop:exact} will serve as starting point of PnP.

\begin{lemma}[Variational\label{lem:noisy} reconstruction]
Let \ref{ass:A1} be satisfies and $\reg\in \Gamma_0(\X)$, then for all $(x_\la^\delta,y^\delta) \in \X \times \Y$ and $s,\la>0$ the following are equivalent:
\begin{enumerate}[label=(\roman*)]
\item\label{lem:noisy1}  $x_\la^\delta  \in \argmin  \dis(\cdot,y^\delta)  +  \la \reg$.
\item\label{lem:noisy2}  $ \nabla_x \dis(x_\la^\delta,y^\delta)  +  \la \nabla_x \reg(x_\la^\delta)  = 0$.
\item\label{lem:noisy3} $x_\la^\delta \in \fix(  \prox_{s \la \reg} (\id_\X - s \nabla_x  \dis(\cdot, y^\delta))$.
\end{enumerate}
\end{lemma}

\begin{proof}
By Fermat's rule,  $x_\la^\delta$ minimizes $\dis(\cdot,y^\delta) + \la \reg$ if and only if  $ \nabla_x\dis(x_\la^\delta, y^\delta)  +  \la \nabla_x\reg (x_\la^\delta) = 0$, which shows  \ref{lem:noisy1} $\Leftrightarrow$  \ref{lem:noisy2}. Moreover,
\begin{align*}
& 0 =  \nabla_x\dis(x_\la^\delta,y^\delta)  +  \la \nabla_x\reg (x_\la^\delta) \notag \\
& \Leftrightarrow  x_\la^\delta + s \la \nabla_x \reg (x_\la^\delta) = x_\la^\delta - s \nabla_x \dis(x_\la^\delta,y)  =\\
& \Leftrightarrow x_\la^\delta = (\id_\X + s \la \nabla_x \reg)^{-1}\circ (\id_\X - s \nabla_x \dis(\cdot,y))(x_\la^\delta) \\
& \Leftrightarrow x_\la^\delta \in \fix(  \prox_{s \la \reg} \circ  (\id_\X - s \la  \nabla_x  \dis(\cdot, y^\delta))
\end{align*}
which gives  the equivalence to \ref{lem:noisy3}.
\end{proof}

Lemma~\ref{lem:noisy} shows that minimizers of the Tikhonov functional are fixed points of the operator $\prox_{s \la \reg} \circ (\id_\X - s \la  \nabla_x  \dis(\cdot, y^\delta))$. Replacing $\prox_{s \la \reg}$ by a general denoiser results in the PnP method as defined next.

\begin{definition}[PnP fixed points]
Under\label{def:pnp} \ref{ass:A1}, \ref{ass:A2} we define the PnP reconstruction  using denoiser  $\Do(\la,\cdot)  \colon \X \to \X$ and step size $s \in (0, 2\beta)$ by $\PnP(\la,\cdot) \colon \Y \to \X$,
 \begin{equation} \label{eq:PnP}
\PnP(\la,y^\delta) \triangleq   \fix ( \Do(\la,\cdot) \circ (\id_\X -  s \nabla_x \dis(\cdot,y^\delta)) ) \,.
 \end{equation}
Note that  the step size is fixed throughout our analysis and therefore not indicated in the notation of  $\PnP(\la,\cdot)$.
 \end{definition}

To compute \eqref{eq:PnP} one may use for any initial value $x_{\la, 0}^\delta \in \X$  the fixed point iteration
 \begin{equation} \label{eq:PnP1}
    x_{\la, n+1}^\delta = \Do(\la,\cdot) \circ (x_{\la, n}^{\delta}  -  s \nabla_x  \dis(x_{\la, n}^\delta, y^\delta)  )  \,.
\end{equation}
In the special case where  $\Do(\la,\cdot) = \prox_{s \la \reg}$,  iteration  \eqref{eq:PnP1} reduced to the FBS  iteration. In the general case, we thus refer to \eqref{eq:PnP1} as PnP-FBS iteration (for $\argmin_x \dis(x,y)$ using denoiser $\Do(\la,\cdot)$).
Convergence of \eqref{eq:PnP1} towards $\PnP(\la,y^\delta)$ can be deduced from Banachs fixed point theorem:

\begin{proposition}[Convergence of PnP-FBS] \label{prop:PnPconv}
If \ref{ass:A1}, \ref{ass:A2} hold and $\Do(\la,\cdot)$ is a contraction,  then  $(x_{\la,n}^\delta)_{n\in \N}$ strongly converge  to $\PnP(\la,y^\delta) $. Moreover,  linear convergence \begin{equation*}
\forall n \in \N \colon \;
	\norm{x_{\la,n+1}^\delta - \PnP(\la,y^\delta)} \leq c \norm{x_{\la,n}^\delta - \PnP(\la,y^\delta)}
\end{equation*}
with some $c>0$ holds.
 \end{proposition}

\begin{proof}
According  to  Lemma~\ref{lem:grad-step}  and \ref{ass:A1}, \ref{ass:A2} and because $\Do(\la,\cdot)$ is contractive, $\Do(\la,\cdot) \circ (\id_\X - s \nabla_x \dis(\cdot,y))$ is contractive.  Hence Banachs fixed point theorem shows that $x_{\la,n}^\delta$ linearly converges to  $\PnP(\la,y^\delta)$, the  unique fixed point  of $\Do(\la,\cdot)  \circ  (\id_\X -  s\nabla_x \dis(\cdot,y^\delta))$.
\end{proof}

Note that the convergence result of Proposition \eqref{prop:PnPconv} is not a novel result  and included here for completeness. More general convergence results are also known. For example, in \cite{Ga21} it is shown that  \eqref{eq:PnP} converges to a fixed point of $\Do(\la,\cdot)  \circ ( \id_\X - s \nabla_x \dis(\cdot,y))$ if $\dis(\cdot,y)$ is convex and Fr\'echet differentiable with Lipschitz continuous gradient, $\Do(\la,\cdot)$ is $a$-averaged for some $a \in (0,1)$, and the fixed point set of $\Do(\la,\cdot)  \circ ( \id_\X - s \nabla_x \dis(\cdot,y))$ is not empty.

\begin{remark}[PnP-ADMM and PnP-PD]
Our theory targets fixed points of $\Do(\la,\cdot) \circ (\id_\X -  s\nabla_x \dis(\cdot,y^\delta)) )$ and therefore apply to any other family of PnP iterations as long as it shares fixed-points with PnP-FBS.
For example, the PnP version of ADMM (alternating direction method of multipliers) reads
\begin{align*}
x_{\la,n+1}^\delta &= \prox_{s \dis(\cdot, y^\delta)}(y_{\la,n}^\delta - z_{\la,n}^\delta) \\
y_{\la,n+1}^\delta &= \Do(\la,x_{\la,n+1}^\delta + z_{\la,n}^\delta)\\
z_{\la,n+1}^\delta &= z_{\la,n}^\delta + x_{\la,n+1}^\delta-y_{\la,n+1}^\delta \,,
\end{align*}
where $y_{\la,0}^\delta, z_{\la,0}^\delta \in \X$ are initializations, $(y_{\la,n}^\delta)_{n\in \N}$,  $(z_{\la,n}^\delta)_{n\in \N}$ are auxiliary sequences and  $(x_{\la,n}^\delta)_{n\in \N}$ is the main  sequence. In \cite[Proposition 3]{Wo19} is shown that if PnP-ADMM is convergent, then $(x_{\la,n}^\delta)_{n\in \N}$ converges to a fixed point of $\Do(\la,\cdot) \circ ( \id_\X - s \nabla_x \dis(\cdot,y))$. Consequently, if the fixed point is unique PnP-FBS and PnP-ADMM converge to the same point.
The same holds also for  primal dual PnP variants (PnP-PD) \cite[Remark 3.1]{Me17}.  The  regularization theory that we develop for PnP therefore applies equally to any of these algorithms.
\end{remark}

\section{PnP as regularization  method}
\label{sec:analysis}

In the following consider a family  $(\Do(\la,\cdot))_{\la > 0}$ of denoisers  each associated with the PnP-FBS iteration and associated PnP limits $\PnP(\la,y^\delta) = \fix ( \Do(\la,\cdot) \circ (\id_\X -  s\nabla_x \dis(\cdot,y^\delta)))$; see Definition~\ref{def:pnp}.

\subsection{Main assumptions}
\label{sec:assumptions}

Our analysis uses the following assumptions on the family  $(\Do(\la,\cdot))_{\la > 0}$.

\begin{definition}[Denoising family] \label{def:denoisers}
We call $(\Do(\la,\cdot))_{\la >0}$   admissible  family  of denoisers if the following hold:
\begin{enumerate} [label=(B\arabic*), leftmargin =2.5em]
\item \label{ass:B1} $ \forall \la>0 \colon  \Do(\la,\cdot) \colon \X \to \X$ is a contraction.
\item \label{ass:B2} $(\Do(\la,\cdot))_{\la>0} \to \id_\X $ strongly point-wise
\item  \label{ass:B3}$(\Do(\la,\cdot))_{\la>0} \to \id$ weakly uniformly on bounded sets.
\item \label{ass:B4} $\exists E \subseteq \X \, \forall  x\in E \colon
\norm{ \Do(\la,x)-x }  = \mathcal{O}(1- \lip(\Do(\la,\cdot)) )$.
\end{enumerate}
\end{definition}

Convergence of PnP-FBS and stability of PnP limits only requires \ref{ass:B1}. It can be seen as our requirement for $\Do(\la,\cdot)$ being a denoiser and is easily to achieve.
A denoiser might only be non-expansive rather than contractive, which is not directly included in our theory. However as we show below this can  be restored  by  a simple scaling trick.

\begin{lemma}[Non-expansive denoisers] \label{lem:scaled}
Let $(\Do(\la,\cdot))_{\la > 0} $ be a family of non-expansive operators satisfying \ref{ass:B2}, \ref{ass:B3} and let $\sigma \colon [0,\infty) \to (0,1]$ be  strictly decreasing,  continuous at zero with $\sigma(0)=1$ and   $ \forall x \in E\colon \norm{\Do(\la,x) - x}  = \mathcal O (1 - \sigma(\la))$ as $\la \to 0$. Then $( \sigma(\la) \Do(\la,\cdot))_{\la > 0}$ satisfies \ref{ass:B1}-\ref{ass:B4}.
\end{lemma}

\begin{proof}
Clearly, $\lip(\sigma(\la)\Do(\la,\cdot)) \leq \sigma(\la)$ and for all $x \in \X$ we have $\sigma(\la)\Do(\la,x) \to x$ strongly, which gives \ref{ass:B1}, \ref{ass:B2}.
Now let $z \in \X$ and $B \subseteq \X$ be bounded by $R>0$, then
\begin{align*}
&\sup_{x \in B} \abs{\inner{\sigma(\la)\Do(\la,x) -x}{z}}
\\& \qquad\leq \sup_{x \in B}
\abs{ \sigma(\la)\inner{ \Do(\la,x) -x}{z} -(1-\sigma(\la)) \inner{x}{z} }
\\& \qquad\leq
\sigma(\la) \sup_{x \in B} \abs{\inner{ \Do(\la,x) -x}{z}} + R (1-\sigma(\la))  \norm{z} \,,
\end{align*}
which converges to zero as $\la \to 0$, showing \ref{ass:B3}.
Finally, for $x \in E$ we have
\begin{align*}
&\frac{\norm{\sigma(\la) \Do(\la,x)-x}}{1-\lip(\sigma(\la)\Do(\la,\cdot))} \\
&\qquad\leq  \frac{\norm{\sigma(\la) (\Do(\la,x)-x) -(1-\sigma(\la)) x}}{1-\sigma(\la)} \\
&\qquad\leq \sigma(\la) \frac{\norm{\Do(\la,x)-x}}{1-\sigma(\la)} + \norm{x} \\
&\qquad\leq C + \norm{x}\,,
\end{align*}
which  is \ref{ass:B4}.
\end{proof}

The additional assumptions \ref{ass:B2}-\ref{ass:B4} are used to establish convergence as $\delta \to 0$.  Condition \ref{ass:B2} is quite  natural and requires the denoising effect to  vanish in the limit. Additionally, the convergence requires some form of uniform convergence condition. Lemma~\ref{lem:lip} shows that assuming uniform convergence on the whole space would be too strong.

\begin{lemma} \label{lem:lip}
There exists no sequence $(\To_k)_{k\in \N}$ of contractions $\To_k \colon \X \to \X$ which converges uniformly to $\id_\X$.
\end{lemma}

\begin{proof}
Assume $(\To_k)_{k\in \N}$  converge uniformly to $\id_\X$ and $\eps> 0$. Then we can choose $N \in \N$ such that $\norm{\To_k x -x} < \eps$ for $k \geq N$ and $x \in \X$ and thus $\{ x \mid  \norm{\To_k x -x} < \eps \} = \X$. In particular, there exists a sequence of elements  $x_\ell$  in this set with  $\norm{x_k-x_\ell} \to \infty$. Further, $\norm{x_k-x_\ell}  \leq  \norm{\To_k(x_k) - x_k} + \norm{\Do_k(x_k) - x_\ell} \leq \eps + \lip(\To_k)\norm{x_k-x_\ell}$. Thus $1  \leq  \eps/\norm{x_k-x_\ell} + \lip(\To_k)  \to \lip(\To_k)$.
This contradicts the assumption that $\To_k$ is a contraction,
\end{proof}

In light of  Lemma~\ref{lem:lip},  Condition \ref{ass:B3} looks only at bounded subsets of $\X$ and, furthermore  considers convergence with respect to the weak topology only. For example, we will show that soft thresholding converges uniformly to $\id$ on bounded sets with respect to the weak topology but not with respect to the norm topology; see Example~\ref{ex:thresh}. The same holds for a scaled version of the soft thresholding operator to make sure that they are contraction mappings.

Condition \ref{ass:B4} is also some kind of uniform convergence condition. In fact, the point-wise convergence in \ref{ass:B2} implies that the Lipschitz constants converge to 1; see Lemma \ref{lem:lip-conv}. Thus  \ref{ass:B4}  means  that $\norm{\Do(\la,x)-x}$ for $\la \to 0$ converges to $0$ at least as fast as $\lip(\Do(\la,\cdot))$ converges to $1$.

\begin{lemma}\label{lem:lip-conv}
Let $(\Do_k)_{k\in \N}$ be a sequence of contractions  $\Do_k \colon \X \to \X$   that converge point-wise to $\Do$ with  $\lip(\Do)  =1$. Then  $\lip(\Do_k) \to 1$.
\end{lemma}

\begin{proof}
By the triangle inequality, we have $\norm{\Do x_1- \Do x_2} \leq \lip(\Do_k)\norm{x_1-x_2} + \norm{\Do x_1- \Do_k x_1} + \norm{\Do_k x_2 -\Do x_2}$. With the point-wise convergence of $(\Do_k)_{k \in \N}$   this shows that $\Do$ is Lipschitz continuous with $\lip(\Do) \leq \liminf_{k \to \infty} \lip(\Do_k)$. Therefore,
\begin{equation*}
	1 = \lip(\Do) \leq \liminf_{k \to \infty} \lip(\Do_k) \leq  \limsup_{k \to \infty} \lip(\Do_k) \leq 1
	\end{equation*}
from which we conclude $\lip(\Do_k) \to 1$.
\end{proof}

\subsection{Examples}

\subsubsection{Proximal denoisers}

For $\reg  \in \Gamma_0(\X) $ consider  minimizers of the Tikhonov functional $\dis(\cdot,y^\delta) + \la \reg$ which are according  to Proposition  \eqref{lem:noisy}  equal to PnP fixed points  with $\Do(\la,\cdot) = \prox_{s\la \reg}$.
In this subsection we show  that variational regularization actually fits in our framework. Further we provide examples where PnP extends variational regularization beyond proximal denoisers.

\begin{lemma} \label{lem:strong}
If  $\reg  \in \Gamma_0(\X) $ is $a$-strongly convex then $\Do(\la,\cdot) = \prox_{s\la \reg}$ is $1 / (1+as\la)$-Lipschitz and  satisfies \ref{ass:B1}, \ref{ass:B2}, \ref{ass:B4}.
\end{lemma}

\begin{proof}
Since $s\la\reg$ is $as\la$-strongly convex, $s\la \partial \reg$ is strongly monotone with constant $as\la$.  With \cite[Prop. 23.13]{bauschke2011convex} this shows that $\prox_{s\la \reg} = (\id+s\la \partial \reg)^{-1}$ is Lipschitz-continuous with constant $1 / (1+as\la)$.  According to     \cite{bauschke2011convex}, $\prox_{s\la \reg} (x) = x -s\la  \partial_0 \reg(x) + o(s\la)$  for $x \in \dom(\reg)$  where $\partial_0 \reg (x)$ denotes the subgradient with minimal norm. Together with   $1- \lip(\prox_{s\la \reg}) \geq  as\la / (1+as\la)$  this gives  $\norm{x-\prox_{s\la \reg}x} = \mathcal O(1- \lip(\prox_{s\la \reg}))$.
\end{proof}

\begin{remark} \label{remark:scale-reg}
If $\reg$ is not strongly convex then $ \prox_{s\la \reg}$ is still non-expansive. To derive a contraction from $ \prox_{s\la \reg}$ one may apply  Proposition \ref{lem:scaled}.  If $\Do(\la,\cdot) = \prox_{s \la\reg}$ satisfies \ref{ass:B3} and $\sigma(\la)=(1-\la)_+$, then
for all $x \in \dom(\reg)$ we have
\begin{multline*}
\frac{\norm{\prox_{s\la \reg}(x)-x}}{1-(1-\la)}
=
\frac{1}{\la}  \norm{s \la  \partial_0 \reg(x)  + o(\la)}
\\ \leq s    \norm{\partial_0 \reg(x)}  + \norm{ o(\la)/\la} \to s   \norm{\partial_0 \reg(x)}
\text{ as $\la \to 0$} \,.
\end{multline*}
Thus  $( (1-\la)_+ \Do(\la,\cdot))_{\la > 0}$ satisfies \ref{ass:B1}-\ref{ass:B4}.
\end{remark}

Assumption \ref{ass:B3}, for example, is satisfied for soft and hard thresholding as the following example shows.

\begin{example}[Thresholding]\label{ex:thresh}
Consider the hard and soft thresholding  operators in $\ell^2(\N)$, $\hard(\la,\cdot), \soft(\la,\cdot) \colon \ell^2(\N) \to \ell^2(\N)$, respectively, defined by $\soft(\la,x)_i = \hard(\la,x)_i =0$ if $\abs{x_i} \leq \la$ and
\begin{align}
\hard(\la,x)_i &= \abs{x_i}   \\
\soft(\la,x)_i &=  \sign{(x_i)}(\abs{x_i} - \la)
\end{align}
otherwise.  Then $(\hard(\la, \cdot))_{\la >0}$, $(\soft(\la, \cdot))_{\la >0}$ converge uniformly to $\id$ on bounded sets with  respect to the weak topology, but not with respect to the norm topology.
\end{example}

\begin{proof}
Let us start with hard thresholding.   Without loss of generality assume that  $z \in \ell^2(\N)$ has  non-negative entries ordered in descending order and that $B$  is the closed centered unit ball. Then $\inner{\hard(\la,x)-x}{z} = -  \sum_{\abs{x_i}  \leq \la } x_i z_i$ and thus
\begin{equation} \label{eq:threshC1}
	\sup_{x \in B} \inner{\hard(\la,x)-x}{z}  = \sup_{x \in C_\la} \inner{x}{z} \,,
\end{equation}
where  $C_\la \triangleq \{x \in B \mid \norm{x}_\infty \leq \la \}$.  Set $C_\la$  is convex and closed and $\argmin_{x\in C_\la} \norm{x-z}^2
=\argmax_{x\in C_\la} \inner{x}{z}$. The projection theorem thus gives a unique $x^*_\la$ maximizing $\inner{x}{z}$ over $C_\la$.  One verifies that $x_\la^*$ is given by
\begin{equation*}
 	\forall i \in \N \colon \quad
	 (x_\la^*)_i =
	\begin{cases}
	\la  & i = 0, \ldots, n^*_\la \\
	a^*_\la z_i  & \text{otherwise } \,,
	\end{cases}
\end{equation*}
where  $a^*_\la$ is such hat $\norm{x_\la^*}=1 $ and $n^*_\la \in \N$ is the smallest natural number with $\norm{x^*_\la}_\infty \leq \la $. We have $n_\la^* +1 \leq 1 / \la^2$, $n_\la^* \to \infty$ as $\la \to 0$ and
\begin{equation} \label{eq:threshC2}
\sup_{x \in C_\la} \inner{x}{z}  =  \sum_{i=0}^{n^*_\la} \la  z_i + a^*_\la \sum_{i=n^*_\la+1}^\infty z_i^2 \\
 \leq \frac{1}{\sqrt{n_\la^* +1}}\sum_{i=0}^{n_\la^*} \abs{z_i} + a^*_\la \sum_{i=n_\la^*+1}^\infty \abs{z_i}^2  \,.
\end{equation}
Let $(e_k)_{k\in \N}$ with $(e_k)_i= 1$ for $k=i$ and  $0$ otherwise be the standard basis of $\ell^2(\N)$, then $s_k \triangleq (k+1)^{-1/2} \sum_{i=0}^{k} e_i  \to 0$ weakly and thus   $1/\sqrt{n^*_\la+1} \sum_{i=0}^{n^*_\la} \abs{z_i} = \inner{s_k}{\abs{z}} \to 0$. With  \eqref{eq:threshC1} and \eqref{eq:threshC2}  we get $\sup_{x \in B} \inner{\hard(\la,x)-x}{z}  \to 0$.

Because  $\sup_{x \in B} \inner{\soft(\la, x)-x}{z}  = \sup_{x \in C_\la} \inner{x}{z}$  the above proof applies to soft thresholding, too.
Finally,
\begin{equation*}
	\sup_{x \in B} \norm{\soft(\la, x)-x}
\\ =\sup_{x \in B} \norm{\hard(\la, x)-x}
 = \sup_{x \in C_\la} \norm{x}=1
\end{equation*}
and therefore both operators do not uniformly converge on bounded sets in the norm topology.
\end{proof}

Soft thresholding is the proximal operator of the $\ell^1$-norm. Thus, according to Example~\ref{ex:thresh} and Remark~\ref{remark:scale-reg}  the scaled thresholding operations
\begin{equation*}
   \forall \la >0\colon \quad   \Do_\la \triangleq (1-\la)_+ \soft(\la, \cdot)
\end{equation*}
satisfy \ref{ass:B1}-\ref{ass:B4} and thus form an admissible family  of denoisers.

\subsubsection{Beyond proximal denoisers}

Not all denoisers are of the proximal type. According to a theorem of Moreau \cite{Mo65}, $\Do \colon \X \to \X$ is proximal if and only if it is non-expansive and the subgradient of a convex functional. If $\Do$ is linear, then $\Do$ is  proximal  if and only if it is self-adjoint and $\norm{\Do} \leq 1$.  Let us provide simple examples satisfying \ref{ass:B1}-\ref{ass:B4}.

\begin{example}[Filter methods] Let $\Uo \colon \X \to \ell^2(\N)$ be unitary and  consider the diagonal operators $\Mo_\la \colon \ell^2(\N) \to \ell^2(\N) $ defined by $(\Mo_\la (x))_i  = m_{\la,i} x_i$ with bounded $m_{\la,i}$. Obviously  $\Do(\la,\cdot) = \Uo \Mo_\la  \Uo^*$ is linear and bounded. Moreover
\begin{itemize}
\item $\Do(\la,\cdot)$ self-adjoint $\Leftrightarrow$
$\forall i \colon m_{\la,i} \in \R$.
\item $\Do(\la,\cdot)$  positive $\Leftrightarrow$
$\forall i \colon m_{\la,i} >0$.
\item $\Do(\la,\cdot)$ contraction $\Leftrightarrow$ $\sup_\la \abs{m_{\la,i}} < 1$.
\end{itemize}
By  Moreau's theorem, $\Do(\la,\cdot)$ is  proximal if and only  if $m_{\la,i} \in [0,1]$. If in addition $m_{\la,i} \in (0, 1- \eps] $ for some $ \eps \in (0,1)$ then $\Do(\la,\cdot)$ is proximal  and contractive.

Linear proximal mappings have to be self-adjoint and positive and it easy to construct linear contractions of the form $\Uo \Lo_\la   \Uo^*$ which are not proximal. For example, $\Lo_\la \colon \ell^2(\N) \to \ell^2(\N) $ defined by $(\Lo_\la (x))_i  = (1-2\la) x_i + \la x_{i-1}$ is a contraction  but not self-adjoint. Therefore $\Do(\la,\cdot) = \Uo \Lo_\la  \Uo^*$ is a contraction but not  proximal.
\end{example}

Consider the discrete convolution operator  $\Ko \colon  \ell^2(\Z) \to \ell^2(\Z)$  with kernel $k \in \ell^1(\Z)$ defined by
\begin{equation} \label{eq:K}
	\Ko x (n) = (k \ast x) (n) \triangleq  \sum_{m \in \Z} k(m) x(n-m) \,.
\end{equation} Write  $(\ft k)(z) \triangleq \sum_{m \in \Z} k(m) z^{-m}$ for $z \in \sph^1$  for  the Fourier transform  of $k$.  Then by Plancherel's theorem,  $\norm{\Ko} =  \norm{\ft k}_\infty$.  In particular, $\Ko$ is bounded if and  only  if $\ft k$ is bounded.   If $k$ is not symmetric   then   $\Ko$ is  not self-adjoint  and therefore not  proximal.  As  specific example for a  family of non-proximal convolutions that satisfy \ref{ass:B1}-\ref{ass:B4} is  presented  in the following  Example~\ref{ex:causal}.

\begin{example}[Causal \label{ex:causal}denoising]
Consider a family of discrete convolutions $\Ko_\la$ of the form \eqref{eq:K} with kernels and associated Fourier transforms
\begin{align}\label{eq:ker}
	k_\la(m) &=  (1-e^{-1/\la}) e^{-m(\la+1)/\la} \, \mathbf{1}_{m\geq 0}    \,,
	\\
	(\ft k_\la)(z) &= \frac{z (1-e^{-1/\la})  }{ z-e^{-1/(1+\la)}} \,.
\end{align}
For  $z \in \ell^2(\Z)$ and bounded $B \subseteq \ell^2(\Z)$ one verifies:
\begin{itemize}
\item $\lip(\Ko_\la) = \norm{\Ko_\la }  = \norm{\ft k_\la}_{\infty}  < 1$.
\item $ \norm{k_\la \ast x - x} \leq  \norm{1-\ft k_\la}_\infty \norm{x} $.
\item  $\sup_{x \in B}  \abs{\inner{k_\la \ast x -x}{z}} \leq \norm{1-\ft k_\la}_\infty \norm{z} \sup_{x \in B} \norm{x}$.
\item $\norm{1-\ft k_\la}_\infty \to 0$ as $\al \to 0$.
\end{itemize}
This gives \ref{ass:B1}-\ref{ass:B3}.  Further, with
\begin{equation*}
\frac{\norm{1-\ft k_\la}_\infty}{1-\norm{\ft k_\la}_{\infty} }
= \frac{1+e}{1+e^{(1+\la)/\la}} \cdot \frac{e^{(1+\la)/\la}-1}{e-1}
\to \frac{e+1}{e-1}
\end{equation*}
we also obtain \ref{ass:B4}.
\end{example}

\subsection{Stability}
\label{sec:stability}

As first theoretical question we answer the stability.  In fact we will give a  quantitative  stability estimate.

\begin{theorem}[Stability \label{thm:stabest} estimate]
Consider~\eqref{eq:min}  with \ref{ass:A1}, \ref{ass:A2}, \ref{ass:A4} and let $\Do(\la,\cdot)$ be  contractive.
Then for all $y_1, y_2 \in \Y$ we have
\begin{equation}\label{eq:stab-estimate}
\norm{ \PnP(\la, y_1) -   \PnP(\la, y_2)}  \leq  \frac{\gamma \lip(\Do(\la,\cdot))}{1- \lip(\Do(\la,\cdot))}  \sup_{x \in \X} \norm{\nabla_x \dis (x,y_1)-\nabla_x \dis (x,y_2)} \,.
\end{equation}
 \end{theorem}

\begin{proof}
The following proof uses ideas from \cite{Na68}. Fix $\la >0$, let $y_1, y_2 \in \Y$  and set $\To_i \triangleq \Do(\la,\cdot) \circ (\id_\X -  s \nabla_x \dis ( \cdot,y_i))$ for  $i=1,2$. Then   $\To_i$ is a contraction  with $ \lip(\To_i) \leq \lip (\Do(\la,\cdot))$. Write $x_i =  \ffix(\To_i) =  \PnP(\la, y_i) $. Then
\begin{align*}
\norm{ x_1 -   x_2}
& = \norm{  \To_1 x_1 -  \To_2  x_2}
\\
& \leq  \norm{  \To_1 x_1 - \To_2  x_1}  +\norm{  \To_2 x_1 -  \To_2  x_2}
\\
&  \leq \norm{  \To_1 x_1 -  \To_2  x_1}  + \lip(\To_2) \norm{   x_1 -  x_2}  \,.
\end{align*}
From the definitions of $\To_1, \To_2$ we derive    $\norm{  \To_1 x_1 -  \To_2  x_1} \leq   \lip (\Do(\la,\cdot)) \norm{\nabla_x \dis (x,y_1)-\nabla_x \dis (x,y_2)}$. Together with the last displayed equation this gives   \eqref{eq:stab-estimate}.
\end{proof}

Theorem~\ref{thm:stability} is a corollary of Theorem~\ref{thm:stabest} that we formulate for the more general Problem \eqref{eq:min}.

\begin{corollary}[Stability of PnP] \label{cor:stab}
Consider Problem~\eqref{eq:min} with \ref{ass:A1}, \ref{ass:A2}, \ref{ass:A4} and let $\Do(\la,\cdot)$ be contractive. Then $\PnP(\la,\cdot)$ is strongly continuous.
\end{corollary}

\begin{proof}
Let  $y, y_k \in \Y$ with $\norm{y_k - y} \to 0$ as $k \to \infty$.  According to the  equicontinuity of $(\nabla_x \dis (x,\cdot)$, \begin{equation*}
	\sup_{x \in \X} \norm{\nabla_x \dis (x,y)-\nabla_x \dis (x,y_n)} \to 0 \text{ as } n \to \infty
	\,.
\end{equation*}
Thus $\norm{ \PnP(\la, y_k) -   \PnP(\la, y)}  \to 0$ by Theorem~\ref{thm:stabest}.
\end{proof}

For the least squares functional $\dis_{\rm LS}(x,y) = \norm{\Ao x-y}^2/2$, \eqref{eq:stab-estimate} becomes
\begin{equation*}
\norm{ \PnP(\la, y_1) - \PnP(\la, y_2)}  \leq  \frac{\gamma \norm{\Ao}   \lip(\Do(\la,\cdot))}{1- \lip(\Do(\la,\cdot))}  \norm{y_1-y_2} \,,
\end{equation*}
which is a linear stability estimate which tends to infinity as $\la \to 0$.
Further note that if $\id_\X - s \nabla_x \dis(\cdot,y)$ would be a contraction then following  Theorem~\ref{thm:stabest} one shows that $\argmin_x \dis(x,y) = 0$  has a unique and stable solution.
This implies that in the ill-posed case, $\id_\X - s \nabla_x \dis(\cdot,y)$ is  no contraction.

\subsection{Convergence}
\label{sec:convergence}

To establish convergence we start with two Lemmas.

\begin{lemma}\label{lem:bound-fix}
Let $\To \colon \X \to \X$ be non-expansive with  fixed point $a$ and  $(\To_k)_{k \in \N}$ be a sequence of contractions on $\X$ with fixed points $a_k$. Then
\begin{equation*}
\left(\frac{a-\To_k(a)}{1- \lip(\To_k)}\right)_{k \in\N} \text{ bounded } \Rightarrow (\ffix(\To_k))_{k \in  \N} \text{  bounded} \,.
\end{equation*}
\end{lemma}

\begin{proof}
Let $M >0 $ with $\norm{a-\To_k(a)}/(1- \lip(\To_k)) \leq  M$. By the triangle inequality, $\norm{a_k -a}  \leq \lip(\To_k) \norm{a_k-a} + \norm{\To_k(a)-a}$ and thus
\begin{equation*}
\norm{a_k -a}
\leq \lip(\To_k)^n \norm{a_k-a}
+ \norm{\To_k(a)-a}  \sum_{i=0}^{n-1} \lip(\To_k)^i  \to  \frac{\norm{\Ao_k(a)-a}}{1-\lip(\Ao_k)} \leq  M \,.
\end{equation*}
Thus  $(a_k)_{k\in \N}$ is bounded by $M + \norm{a}$.
\end{proof}

\begin{lemma}\label{lem:weak-conv}
Let  $\To \colon \X \to \X$ be weakly sequentially continuous and let
$(\To_k)_{k \in \N}$ be a sequence of non-expansive mappings on $\X$ with fixed points $a_k$. If $(a_k)_{k \in \N}$ is contained in some bounded set $B \subseteq \X$ and $(\To_k)_{k \in \N} \to \To$ weakly uniformly on $B$, then  the limit of every weakly convergent subsequence of $(a_k)_{k \in \N}$  is a fixed point of $\To$.
\end{lemma}

\begin{proof}
Let $(a_{\tau(k)})_{k\in \N}$ be a weakly convergent subsequence of $(a_k)_{k \in \N}$ with  limit $a$. Then
\begin{equation*}
\inner{a_{\tau(k)} - \To(a_{\tau(k)})}{z}
 =  \inner{\To_{\tau(k)}(a_{\tau(k)}) - \To(a_{\tau(k)}) }{z}
 \leq \sup_{x \in B } \inner{\To_{\tau(k)}(x) - \To(x)}{z} \,.
\end{equation*}
By the  weak uniform convergence of $\To_k$ the right hand side tends to 0. Together with the weak  sequential  continuity of $\To$ we get  $ \inner{a - \To(a)}{z} \leq 0 $. Considering $-z$ in place of  $z$, we conclude   $\forall z \in \X \colon \inner{a - \To(a) }{z} =  0$ and thus $\To(a) = a$. \end{proof}

We  next establish convergence of PnP. Again we formulate and verify the result for the  more general Problem~\ref{eq:min}.

\begin{theorem}[Convergence] \label{thm:conv}
Consider Problem~\eqref{eq:min}, let \ref{ass:A1}-\ref{ass:A4} hold  and $(\Do(\la,\cdot))_{\la >0}$ be an admissible family  of denoiser, let $y \in E$,  $(\delta_k)_{k \in \N} \in (0,\infty)^\N$  converge to $0$. Then there  exists $(\la_k)_{k \in \N} \in (0,\infty)^\N$  converging to $0$  such that for all $y_k \in \Y$ with $\norm{y-y_k} \leq \delta_k$ the following hold:
\begin{enumerate}[label=(\alph*)]
	\item $(\PnP(\la_k,y_k))_{k\in \N}$ has a weakly convergent subsequence.
	\item  The limit of every weakly convergent subsequence of $(\PnP(\la_k,y_k))_{k\in \N}$ is a solution  of $\argmin_x \dis(x,y)$.
\end{enumerate}
\end{theorem}

\begin{proof}
Set $\eta_k \triangleq \sup_{x \in \X} \norm{\nabla_x \dis(x,y_k) - \nabla_x \dis (x, y)}$. By the equicontinuity of $(\nabla_x \dis(x,\cdot))_{x \in \X}$, we have $ \eta_k \to 0$ as $k \to \infty$.
Because $\lip(\Do(\la,\cdot)) \to 1$, there exists a sequence $\la_k>0$ with $\la_k  \to 0$ and  $\lip(\Do(\la_k,\cdot)) \leq M/(M+\eta_k)$ for $k$ sufficiently large and fixed $M > 0$.  Then, with $\Do_k \triangleq \Do(\la_k,\cdot)$,
\begin{equation} \label{eq:bound-parameter}
\frac{\gamma \lip(\Do_k)}{1-\lip(\Do_k)} \sup_{x \in \X} \norm{\nabla_x \dis(x,y_k) - \nabla_x \dis (x, y^*)}
\leq \frac{M/(M+\eta_k)}{1-M/(M+\eta_k)} \ga \eta_k = \ga M \,.
\end{equation}
Let $x^*$ be  a minimizer  of   $\dis(\cdot,y^*)$ and write
\begin{align*}
	 \To_k &\triangleq \Do_k \circ ( \id_\X - s \nabla_x \dis(\cdot,y_k))
	 \\
	 \To &\triangleq \id_\X - s \nabla_x \dis (\cdot, y^*) \,.
\end{align*}
Then $x^*$  is a fixed point of  $\To$ and
\begin{align*}
&\norm{\To_k(x^*) - x^*}\\
& =  \norm{\Do_k(x^* - s \nabla_x \dis(x^*,y_k)) - x^*} \\
& \leq \norm{\Do_k  (x^* - s \nabla_x \dis(x^*,y_k)-\Do_k(x^*)}  + \norm{\Do_k(x^*)-x^*} \\
& \leq \ga \lip(\Do_k)  \norm{\nabla_x \dis(x^*, y_k)} + \norm{\Do_k(x^*)-x^*}\\
& = \ga \lip(\Do_k)  \norm{\nabla_x \dis(x^*, y_k) - \nabla_x \dis (x^*, y^*)} + \norm{\Do_k(x^*)-x^*} \\
& \leq \ga \lip(\Do_k)  \sup_{x\in\X} \norm{\nabla_x \dis(x,y_k) - \nabla_x \dis (x, y)}
+ \norm{\Do_k(x^*)-x^*}.
\end{align*}
With \eqref{eq:bound-parameter} and \ref{ass:B4}  we conclude that  $
	(\norm{\To_k(x^*) - x^*}/(1- \lip(\To_k))_{k \in\N}$
is  bounded.  According to Lemma  \ref{lem:bound-fix},  the sequence $(x_k)_{k\in \N}$ of fixed points of $\Do_k \circ (\id_\X -  s \nabla_x \dis(\cdot,y_k))$ is contained  in some  bounded set $B$. In particular, $(x_k)_{k\in \N}$ has a  weakly convergent subsequence.

For any  $z \in \X$ we have
\begin{align*}
\sup_{x \in B}  &\inner{\To_k( x ) - \To(x)}{z} \\
& \quad \leq  \sup_{x \in B} \inner{\To_k( x ) -  \Do_k\To( x ) }{z} +  \sup_{x \in B}  \inner{ \Do_k\To( x )  - \To(x)}{z}
 \\  &\quad \leq  \ga \sup_{x \in B} \norm{\nabla_x \dis(x,y_k)- \nabla_x \dis (x, y)}\norm{z} + \sup_{h \in \To (B)} \inner{\Do_k(h)-h}{z}  \,.
\end{align*}
The latter quantity  converges to zero as $k \to \infty$  and therefore $(\To_k)_{k \in \N}$ converges weakly uniformly to $\To = \id_\X - \nabla_x \dis (\cdot, y^*)$ on $B$. Application of Lemma  \ref{lem:weak-conv}  thus shows that any  weakly convergent subsequence of $(x_k)_{k\in \N}$   is a fixed point of  $\To$  and thus a minimizer  of   $\dis(\cdot,y^*)$.
\end{proof}

\subsection{Characterization of limiting solutions}
\label{sec:char}

Variational regularization has the property that it not only converges to any solution  but to an  $\reg$-minimizing solution of the  inverse problem at hand. Theorem~\ref{thm:char} provides such a characterization for PnP regularization.  A more complete statement that we will verify  below reads as follows:

\begin{theorem}[Limiting solutions of PnP]\label{thm:char2}
Consider the situation of Theorem \ref{thm:conv}, let $B \subseteq \X$ be a bounded set with  $\PnP(\la_k, y_k) \in B$ and additionally assume
\begin{enumerate}[label=(C\arabic*), leftmargin =2.5em]
\item \label{LP:1} $\dis(x,y)= \norm{\Ao x -y }^2/2$ for  $\Ao \colon \X \to \Y$ bounded  linear.
\item \label{LP:2} $\forall x \in B$  $\forall \la > 0 \colon$  $ \Do(\la,\cdot)^{-1}(x) $ is a singleton.
\item \label{LP:3}  $( \Do(\la,\cdot)^{-1}-\id)/\al$ converges weakly uniformly on $B$ to some weakly continuous $\Ho \colon \X \to \X$.
\end{enumerate}
Then any weak accumulation point $x^*$ of $(\PnP(\la_k,y_k))_{k\in \N}$  satisfies
\begin{align}
	\Ao (x^*) &= y \,, \\
	\Ho(x^*)  &\in  \ker(\Ao)^\perp \,.
\end{align}
\end{theorem}

\begin{proof}
Theorem  \ref{thm:conv} states that any  accumulation point $x^*$  of $(\PnP(\la_k,y_k))_{k\in \N}$ satisfies $\Ao (x^*) = y$.  In order to show Theorem~\ref{thm:char2} we will verify that these limits additionally satisfy  $\Ho(x^*)  \in  \ker(\Ao)^\perp$.
Without loss  of generality assume that $x_k \triangleq \PnP(\la_k,y_k)$ is weakly  convergent with limit $x^*$. For $z \in \X$ set $a_k(z) \triangleq \inner{(\Do(\la_k,\cdot)^{-1}(x_k) - x_k)/\la_k}{z}  $. Then
\begin{align*}
&\abs{a_k(z)-\inner{\Ho(x^*)}{z}}
\\ &
\leq \migabs{\miginner{\frac{\Do(\la_k,\cdot)^{-1}(x_k) - x_k}{\la_k}- \Ho(x_k)}{z} } + \abs{\inner{\Ho(x_k)-\Ho(x^*)}{z}}
\\&
\leq \sup_{x \in B} \migabs{\miginner{\frac{\Do(\la_k,\cdot)^{-1}(x) - x}{\la_k}- \Ho(x)}{z}} + \abs{\inner{\Ho(x_k)-\Ho(x^*)}{z}} \,.
\end{align*}
By \ref{LP:3} and the weak sequentially continuity of $\Ho$ this shows   $\lim_{k \to \infty} a_k(z) = \inner{\Ho(x^*)}{z}$ for all $z \in \X$. Further,
\begin{align*}
\Do(\la_k,\cdot)^{-1}(x_k) - x_k
&=
(\id_\X - s \nabla_x \dis(\cdot,y_k) ) (x_k) - x_k
\\&= - \ga \Ao^*(\Ao x_k - y_k) \in \im(\Ao^*) = \ker(\Ao)^\bot \,.
\end{align*}
Hence,  $a_k(z) =0$  and $\inner{\Ho(x^*)}{z}= \lim_{k \to \infty} a_k(z) = 0 $  for all $ z \in \ker(\Ao)$. Therefore $\Ho(x^*) \in  \ker(\Ao)^\perp$.
\end{proof}

\begin{example}[Relation to $\reg$-minimizing solutions]
Let $\reg \in \Gamma_0( \X)$ be differentiable and $L(y) \triangleq  \set{x \in \X \mid \Ao x =y}$.  Minimizers of  $\reg|_{L(y)}$ are refered to as  $\reg$-minimizing solution of $\Ao x = y$. The first order optimality condition in this case reads    $\nabla \reg (x) \in  \ker(\Ao)^\perp$. In terms of proximity operators,
\[
	\nabla \reg(\hat x)
	= (\prox_\reg^{-1} - \id)(\hat x)
	=  \frac{\prox_{\la \reg}^{-1} - \id}{\la}( x^\ast) \,.
\]
Thus Theorem~\ref{thm:char2} can be applied with $\Do(\la,\cdot) = \prox_{\la \reg}$  and  $\Ho = \nabla \reg$ and  one recovers that accumulation points of $(\PnP(\la_k,y_k))_{k\in \N}$ are $\reg$-minimizing solution of $\Ao x = y$.
\end{example}

\section{Discussion and outlook}

In this  work we extended the PnP  framework to a convergent regularization method for solving ill-posed image reconstruction problems.
In particular, we showed that if the noise tends to zero, PnP fixed points converge to exact solutions of $\argmin_x \dis(x,y)$.  Our results generalize many other regularization techniques, for example, variational regularization  and linear filter methods for compact operators using SVD.   Our theory probably also contains  many other cases (such as non-linear filter methods), where the convergence is barley analyzed. Part of future research is to make use of this general tool to establish a convergence analysis for non-linear filter methods like \cite{Fr19}.
Further objectives are convergence  with respect to the norm-topology and convergence rates.

\end{document}